\documentclass[11pt]{article}
\usepackage{amsmath}
\usepackage{amsthm}
\usepackage{amssymb,amsfonts,mathabx}
\usepackage{latexsym}
\usepackage{todonotes}
\usepackage{nicefrac}
\usepackage{epsfig}
\usepackage{stmaryrd}
\usepackage{setspace}
\usepackage{enumerate}
\usepackage{bbm,ifpdf,tikz}
\ifpdf
\usepackage{pdfsync}
\fi

\usepackage[pagebackref]{hyperref}

\newtheorem{theorem}{Theorem}[section]

\newtheorem{lemma}[theorem]{Lemma}
\newtheorem{proposition}[theorem]{Proposition}
\newtheorem{corollary}[theorem]{Corollary}

\newtheorem{conjecture}[theorem]{Conjecture}

{
\theoremstyle{definition}

\newtheorem{example}[theorem]{Example}

\newtheorem{remark}[theorem]{Remark}
}

\newcommand{\excise}[1]{}
\newcommand{\Spec}{\operatorname{Spec}}

\newcommand{\id}{\operatorname{id}}

\newcommand{\codim}{\operatorname{codim}}
\renewcommand{\dim}{\operatorname{dim}}

\newcommand{\rk}{\operatorname{rk}}

\renewcommand{\and}{\qquad\text{and}\qquad}

\newcommand{\Hom}{\operatorname{Hom}}

\newcommand{\actson}{\reflectbox{$\righttoleftarrow$}}

\newcommand{\GIT}{/\hspace{-3pt}/}

\newcommand{\Z}{\mathbb{Z}}
\newcommand{\Q}{\mathbb{Q}}

\newcommand{\C}{\mathbb{C}}

\newcommand{\Fq}{\mathbb{F}_q}

\newcommand{\IC}{\operatorname{IC}}
\newcommand{\HH}{\operatorname{H}} 
\newcommand{\IH}{\operatorname{IH}}

\newcommand{\Fr}{\operatorname{Fr}}

\renewcommand{\a}{\alpha}

\newcommand{\Aut}{\operatorname{Aut}}

\renewcommand{\deg}{\operatorname{deg}}

\newcommand{\F}{\mathbb{F}}
\newcommand{\beq}{\begin{eqnarray*}}
\newcommand{\eeq}{\end{eqnarray*}}

\newcommand{\cO}{\mathcal{O}}
\newcommand{\Gm}{\mathbb{G}_m}

\newcommand{\SL}{\operatorname{SL}}

\begin{document}
\spacing{1.2}
\noindent{\Large\bf Intersection cohomology of Popov--Vinberg varieties
}\\

\noindent{\bf Andrew Dancer}\\
Jesus College and Mathematical Institute, University of Oxford
\vspace{.1in}

\noindent{\bf Johan Martens}\\
School of Mathematics and Maxwell Institute, University of Edinburgh
\vspace{.1in}

\noindent{\bf Nicholas Proudfoot}\footnote{Supported by NSF grant DMS-2344861.}\\
Department of Mathematics, University of Oregon\\

{\small
\begin{quote}
\noindent {\em Abstract.}  The Popov--Vinberg variety of a simply connected, split, semisimple algebraic group $G$ is a singular affine variety that contains
the basic affine space $G/U$ as a Zariski open subset.  It is defined as the spectrum of the ring of functions on $G/U$, and can also be identified
with the universal symplectic implosion for the maximal compact subgroup of $G$.  We provide a recursive procedure for computing the intersection cohomology
of this variety, with an emphasis on the case where $G = \SL_n$.
\end{quote} }

\section{Introduction}
{\bf Background.}  Let $G$ be a simply connected, split, semisimple algebraic group over a field $\mathbb{F}$, and let $U\subset G$ be a maximal unipotent subgroup.
The quotient $G/U$ is a quasi-affine variety, and if $\mathbb{F}$ has characteristic zero, its coordinate ring includes exactly one copy of every finite dimensional irreducible representation of $G$.
Despite the fact that $G/U$ is not affine, it is known as the {\bf basic affine space}; see \cite{BGGfund, BGG, BraKaz, Polishchuk, GinRic, GurKaz} for a small sample of the literature on this space.\footnote{Sometimes the word ``basic'' is replaced
by ``base'' or ``fundamental''.}
Even though $U$ is not reductive, the coordinate ring $\cO(G/U) \cong \cO(G)^U$ is finitely generated, and we may therefore consider
the GIT quotient $X_G := G\GIT U = \Spec\cO(G/U)$.   This is an affine variety that contains $G/U$ as a Zariski open subset.
It was first studied by Popov and Vinberg \cite{vinberg-popov}, so we call it the {\bf Popov--Vinberg variety} of $G$.
It is singular unless $G$ is a power of $\SL_2$, and the orbits of $G$ in $X_G$ are in bijection with parabolic subgroups that contain $U$. 

When $\mathbb{F}=\C$, the Popov--Vinberg variety is also called the {\bf universal symplectic implosion} for the maximal compact subgroup $K\subset G$. Symplectic implosion may be viewed as an abelianisation construction in symplectic geometry:
given a symplectic manifold $M$ with a Hamiltonian action of a compact group $K$, 
one produces a new space $M_{\rm impl}$ with an action of the maximal torus of $K$,
with the property that the symplectic reductions of $M$ by $K$ and $M_{\rm impl}$ by the torus agree.
The space $M_{\rm impl}$ is called the {\bf implosion} of $M$.  General implosions can be recovered from the implosion of $T^*K$ in a simple manner. (As discussed in \cite{DKM}, this can be viewed as a composition in the real version of the Moore--Tachikawa category.) For this reason, the implosion of $T^*K$ is called ``universal'', and it is isomorphic to $X_G$ as a stratified K\"ahler variety \cite{GJS}.\\
\\
{\bf Results.}  Our main result (Theorem \ref{general group}) is a recursive procedure for computing the intersection cohomology Poincar\'e polynomial of 
the Popov--Vinberg variety of a complex group.  We focus particularly on the case where $G=\SL_n$, in which case our recursion only involves
groups of the form $\SL_m$ for $m\leq n$ (Corollary \ref{second recursion}).  The only non-trivial calculation of this type that has appeared in the literature
before is that of $\IH^*(X_{\SL_3})$ \cite{HJ}; we provide an appendix that lists the intersection cohomology Poincar\'e polynomials
of $X_{\SL_n}$ for $n\leq 13$.
We also prove that the $(2i)^\text{th}$ intersection cohomology Betti number of $X_{\SL_n}$ is a polynomial in $n$ of degree at most $i/2$
(Proposition \ref{growth}), and we conjecture that this polynomial is a non-negative linear combination of binomial coefficients
(Conjecture \ref{binomial}).

We then consider the generating series $\Psi(t,u)$,
where the coefficient of $u^n$ is equal to the Hilbert series
of the $\SL_n$-equivariant intersection cohomology of $X_{\SL_n}$.
Our recursion translates
into a surprisingly simple functional equation\footnote{See Remark \ref{unique}
for a precise statement about the extent to which this equation determines $\Psi(t,u)$.} (Proposition \ref{functional}):
$$\Psi(t^{-1},u)\Psi(t,-u) = 1.$$
One reason for studying this generating series is that the direct sum over all $n$ of the equivariant intersection homology\footnote{Replacing equivariant cohomology with equivariant homology does not change the dimensions.} groups of $X_{\SL_n}$
naturally forms an algebra
with $\Psi(t,u)$ as its Hilbert series; we define this algebra and initiate its study in Section \ref{sec:rings}.\\
\\
{\bf Methods.}  Our proof of Theorem \ref{general group} proceeds by lifting the complex group to a group scheme over $\Z$, base changing to a field of positive characteristic, and employing the Grothendieck--Lefschetz trace formula
to compute the $l$-adic \'etale intersection cohomology of the intersection cohomology (IC) sheaf.  By a standard comparison result (e.g. \cite[Proposition 10.4.1(i)]{Kirwan-Woolf}), the resulting Poincar\'e
polynomial is the same as that of the topological intersection cohomology of the Popov--Vinberg variety over the complex numbers.
These methods are very close to those used to compute Poincar\'e polynomials of stalks of the IC sheaves of
Schubert varieties \cite{KL80}, toric varieties \cite{DL,Fieseler}, hypertoric varieties \cite{PW}, and arrangement Schubert varieties \cite{EPW}.
In each of those cases, the polynomials computed are examples of Kazhdan--Lusztig--Stanley (KLS) polynomials \cite{Stanley,KLS}, and this case is no different.  We are in effect computing KLS polynomials of the Boolean poset of subsets of simple roots of $G$, with respect
to an exotic rank function and kernel.  This is
not a perspective that we pursue, but it is implicit when we invoke the main result
of \cite{KLS} to prove Theorem \ref{general group}.

\vspace{\baselineskip}
\noindent
{\bf Acknowledgments.}
We thank Max Alekseyev, Vladimir Dotsenko, Anton Mellit, Hjalmar Rosengren, and Bal\'azs  Szendr{\H{o}}i for their contributions to Remark \ref{Balazs} via MathOverflow \cite{MOF};
Paul Balister for his help with Remark \ref{Stirling}; and Andr\'e Henriques for insightful comments about the ring $\widehat{R}$. NP thanks All Souls College for its hospitality during the preparation of this document.  
Finally, we thank the referees for their feedback.
For the purpose of open access, the authors have applied a CC BY public copyright licence to any author-accepted manuscript arising from this submission.

\section{The Popov--Vinberg variety}\label{sec:strat}
We fix a field $\F$.
Let $G$ be a simply connected, split, semisimple algebraic group over $\F$, let $T\subset B\subset G$ be a maximal torus and a Borel subgroup, and 
let $U := [B,B]$ be the corresponding maximal unipotent subgroup.  
Let $$X_G := \Spec\cO(G/U)$$
be the associated Popov--Vinberg variety.  The action of $G$ on $X_G$ has finitely many orbits, indexed by parabolic subgroups $B\subset P\subset G$, and the stabilizer of the orbit indexed by $P$ is equal to the commutator $[P,P]$.  Let us make this more concrete, following \cite[Section 6]{GJS}.

Let $\Lambda^* := \Hom(T,\Gm)$ be the weight lattice, and let $\Phi=\Phi^+\cup \Phi^-\subset \Lambda^*$ be the roots.
Consider the {\bf simple roots} $\{\alpha_1, \dots, \alpha_{\rk G}\}\subset\Phi$ and the {\bf fundamental weights} $\{\varpi_1,\ldots,\varpi_{\rk G}\}\subset\Lambda^*$. 
For any subset $S\subset\Delta:= \{1,\ldots,\rk G\}$, we denote by $\Phi_S\subset \Phi$ the set $\{\alpha\in\Phi \mid (\varpi_i,\alpha)=0 \ \text{for all}\ i\in S\}$.  
Consider the decomposition
$$\mathfrak{g}=\mathfrak{t}\oplus \bigoplus_{\alpha\in \Phi} \mathfrak{g}_{\alpha}$$ of the Lie algebra of $G$ into root spaces.
We have the following subgroups of $G$:
\begin{itemize}
\item A parabolic subgroup $P_S$, with Lie algebra $\mathfrak{p}_S=\mathfrak{t}\oplus \displaystyle\bigoplus_{\alpha \in \Phi_S\cup\Phi^+}\mathfrak{g}_\alpha$.
\item The unipotent radical $U_S\subset P_S$, with Lie algebra $\displaystyle\bigoplus_{\alpha\in \Phi^+\setminus \Phi^+_S}\mathfrak{g}_{\alpha}$.
\item The Levi subgroup $L_S\subset P_S$ with root system $\Phi_S$, which has the property that $P_S$ is the semidirect product of $U_S$ and $L_S$.
\item The commutator subgroup $G_S := [L_S,L_S]$.
\item The commutator subgroup $H_S := [P_S,P_S]$, which is the semidirect product of $U_S$ and $G_S$.
\end{itemize}
Note that, with our conventions, as the set $S\subset\Delta$ becomes larger, $\Phi_S$ and $P_S$
both become smaller.  In particular,
when $S=\Delta$, we have $P_\Delta = B$, $H_\Delta = U_\Delta = U$, $L_\Delta = T$, and $G_\Delta$ is the trivial group.
At the other extreme, when $S=\emptyset$, we have $G_\emptyset = H_\emptyset = L_\emptyset = P_\emptyset = G$, and $U_\emptyset$ is trivial.
The orbit of $G$ in $X_G$ corresponding to the subset $S$ is isomorphic to
$G/H_S$.  The open orbit corresponding to $S = \Delta$
is equal to $G/U$, and the closed orbit corresponding to $S=\emptyset$ is
a single point.

\begin{example}\label{Sln1}
Suppose that $G = \SL_n$ and $B\subset G$ is the subgroup of upper triangular matrices.
We have a canonical isomorphism $\Lambda^*\cong \Z^n/\Z$.
Let $e_i$ denote the image in $\Lambda^*$ of the $i^\text{th}$ standard basis vector for $\Z^n$, and let $\a_i = e_i-e_{i+1}$ be the corresponding
simple root.  

There is a bijection between compositions
(ordered partitions) of $n$ and subsets of $\Delta$
given by sending an $r$-tuple $\sigma = (\sigma_1,\ldots,\sigma_r)$ of positive integers with $\sigma_1+\cdots+\sigma_r=n$
to the subset $$S_\sigma := \{{\sigma_1}, {\sigma_1+\sigma_2}, \ldots, {\sigma_1+\cdots+\sigma_{r-1}}\}\subset\Delta.$$
Given a composition $\sigma$, we will write $P_\sigma := P_{S_\sigma}$, and similarly for $L_\sigma$, $G_\sigma$, and $H_\sigma$.
We then have the following explicit descriptions:
\begin{itemize}
\item The parabolic subgroup $P_\sigma$ consists of block upper triangular matrices of determinant 1, and has dimension $\frac{1}{2} \left( n^2 -2 + \sum_{i=1}^{r} \sigma_i^2 \right)$.
\item The unipotent radical $U_\sigma\subset P_\sigma$ consists of block upper triangular matrices whose diagonal blocks are identity matrices, and has dimension
$\frac{1}{2} \left(n^2 -\sum_{i=1}^{r} \sigma_i^2 \right)$.
\item The Levi subgroup $L_\sigma\subset P_\sigma$ consists of block diagonal matrices of determinant 1.
\item The commutator subgroup $G_\sigma := [L_\sigma,L_\sigma] \cong \SL_{\sigma_1}\times\cdots\times\SL_{\sigma_r}$ consists of block diagonal matrices
whose individual blocks each have determinant 1.
\item The commutator subgroup $H_\sigma := [P_\sigma,P_\sigma]$ consists of block upper triangular matrices whose diagonal blocks each have determinant 1.
\end{itemize}
The open orbit corresponds to the composition $\sigma = (1,\ldots,1)$, and
has dimension $\frac{1}{2}n(n+1)-1$, while the unique fixed point corresponds to the composition $\sigma=(n)$.
\end{example}

Let $V_\varpi$ be the irreducible representation of $G$ with highest weight $\varpi$, and let $v_\varpi\in V_\varpi$
be a nonzero highest weight vector.  Let
$$E_G := V_{\varpi_1}\oplus\cdots\oplus V_{\varpi_{\rk G}}.$$
For any subset $S\subset\Delta$, let $$v_S := \sum_{i\in S} v_{\varpi_i}\in E_G,$$
and note that the stabilizer of $v_S$ is equal to $H_S$.  Let $O_S := G\cdot v_S \cong G/H_S$.
We have $H_\Delta = U$, thus there is a unique $G$-invariant map from $G/U$ to $E_G$ taking the identity coset $e\cdot U\in G/U$ to the vector $v_\Delta$.  
Since $E_G$ is affine, this map factors through the Popov--Vinberg variety $X_G$.  

\begin{example}
When $G=\SL_n(\F)$, we have $V_{\varpi_i} = \bigwedge^i(\F^n)$ and $v_{\varpi_i} = e_1\wedge\cdots\wedge e_i$.
The map from $\SL_n/U$ to $E_{\SL_n}$ sends the coset of the identity to the vector
$$v_{\Delta} = e_1 + e_1\wedge e_2 + \cdots + e_1\wedge e_2 \wedge\cdots\wedge e_{n-1}.$$
\end{example}

\begin{remark}
If we wanted to work with groups that are not simply connected, we would need to replace $\Delta$
with a choice of a collection of dominant weights that generates the weight lattice.  To simplify the exposition,
we restrict our attention to simply connected groups, for which the collection of fundamental weights 
is a canonical such choice.
\end{remark}

The following theorem is proved in \cite[Theorem 6.11]{GJS}.

\begin{theorem}\label{gjs-strat}
Suppose that $\F=\C$.
The induced map from $X_G$ to $E_G$ is a closed embedding, thus we may identify $X_G$ with a closed subvariety of $E_G$.
For any subset $S\subset\Delta$, we have $v_S\in X_G\subset E$, and we have a $G$-equivariant stratification
\begin{equation}\label{stratification}X_G\; = \bigsqcup_{S\subset\Delta} O_S.\end{equation}
\end{theorem}

For any subset $S\subset \Delta$, consider the Popov--Vinberg variety $X_{G_S}$ for the group $G_S$.  Since $G_S\cap U$ is a maximal unipotent subgroup of $G_S$,
there is a unique $G_S$-equivariant map $\varphi_S: X_{G_S}\to X_G$ taking the identity coset $e\cdot (G_S\cap U)$ to $v_\Delta$.
Equivalently, if we identify $X_{G_S}$ with the GIT quotient $H_S\GIT U :=\Spec(\F[H_S]^U)$, the map $\varphi_S:H_S\GIT U\to G\GIT U$
is induced by the inclusion of $H_S$ into $G$.
Furthermore, we have a map $G \times X_{G_S}\to X_G$ taking $(g,x)$ to $g\cdot \varphi_S(x)$, and this descends to a map
\begin{equation*}\label{nbhd}\psi_S:(G \times X_{G_S})/H_S\to X_G.\end{equation*}
The following lemma is proved in \cite[Lemma 6.13]{GJS}.

\begin{lemma}\label{slice}
Suppose that $\F=\C$.
The map $\psi_S$ is an open embedding with image
$$\bigsqcup_{S\subset S'\subset \Delta} O_{S'}.$$ 
The map $\varphi_S$ is a closed embedding that exhibits $X_{G_S}$ as a $G_S$-equivariant normal slice to the orbit $O_S\subset X_G$, taking the point $0\in X_{G_S}$ 
to the point $v_S\in X_G$.
\end{lemma}

\begin{example}\label{Sln2}
Continuing with Example \ref{Sln1}, write $X_n := X_{\SL_n}$ and $$X_\sigma := X_{(\SL_n)_{S_\sigma}} \cong X_{\SL_{\sigma_1}\times\cdots\times\SL_{\sigma_n}}
\cong X_{\sigma_1}\times\cdots \times X_{\sigma_r}.$$
The closed inclusion $\varphi_\sigma:X_\sigma\to X_n$
is induced by the block diagonal inclusion of the subgroup $G_\sigma = \SL_{\sigma_1}\times\cdots \times\SL_{\sigma_r}$ into $\SL_n$.
If we take \beq 
E_{\sigma_1} &=& \bigoplus_{i=1}^{\sigma_1-1} \textstyle\bigwedge^{\!i} \F\{e_1,\ldots,e_{\sigma_1}\}\\
E_{\sigma_2} &=& \bigoplus_{i=1}^{\sigma_2-1} \textstyle\bigwedge^{\!i} \F\{e_{\sigma_1+1},\ldots,e_{\sigma_1+\sigma_2}\}\\
&\vdots&\\
E_{\sigma_r} &=& \bigoplus_{i=1}^{\sigma_r-1} \textstyle\bigwedge^{\!i} \F\{e_{\sigma_1+\cdots+\sigma_{r-1}+1},\ldots,e_{n}\},
\eeq
then $\varphi_\sigma$ is the restriction of the map $E_{\SL_{\sigma_1}} \times\cdots\times E_{\SL_{\sigma_r}}\to E_{\SL_{n}}$ that sends $(a_1,\ldots,a_r)$ to
$$a_1 + v_{\varpi_{\sigma_1}}\wedge (1+a_2) + v_{\varpi_{\sigma_1+\sigma_2}}\wedge (1+a_3) + \cdots + v_{\varpi_{\sigma_1+\cdots+\sigma_{n-1}}}\wedge (1+a_r).$$ 
In particular the element $0\in X_\sigma$ is sent to the element $$v_{\varpi_{\sigma_1}} + v_{\varpi_{\sigma_1+\sigma_2}} + \cdots + v_{\varpi_{\sigma_1+\cdots+\sigma_{n-1}}} = v_{S_\sigma}\in X_{\SL_n}.$$
\end{example}

The following technical lemma will be needed for the proof of Theorem \ref{general group}.

\begin{lemma}\label{slice actions}
For any $G$ and any subset $S\subset \Delta$, there exists an action of $\Gm$ on $X_G$
such that the normal slice map $\varphi_S:X_{G_S}\to X_G$ is $\Gm$-equivariant with respect to the scalar action of $\Gm$ on $X_{G_S}$ and the chosen
action of $\Gm$ on $X_G$.  In particular, this action fixes the point $v_S = \varphi_S(0)\in X_G$.
\end{lemma}

\begin{proof}
For each cocharacter $\rho:\Gm\to T$, consider the homomorphism $\lambda_\rho:\Gm\to\Aut(E_G)$ given by letting
$\Gm$ act on $V_{\varpi_i}$ with weight $\langle \rho,\varpi_i\rangle$.  Note that this is not the same as the action
of $\Gm$ on $E_G$ given by composing $\rho$ with the action of $G$ on $E_G$, however these two actions do agree on
the highest weight vectors $v_{\varpi_i}\in E_G$, and therefore on the element $v_{S'}\in E_G$ for any $S'\subset\Delta$.
In addition, the action of $\Gm$ via $\lambda_\rho$ commutes with the action of $G$.
In particular, for any $t\in\Gm$ and $g\in G$, we have
$$\lambda_{\rho}(t)\cdot(g\cdot v_{\Delta})=g\cdot(\lambda_{\rho}(t)\cdot v_{\Delta})=g\cdot (\rho(t)\cdot v_{\Delta})=(g\, \rho(t))\cdot v_{\Delta}\in G/U\subset E_G,$$
therefore our action of $\Gm$ on $E_G$ restricts to an action on $X_G$.

We will take $\rho$ to be the unique cocharacter with the property that $\langle \rho, \varpi_i\rangle = 0$ if $i\in S$ and $1$
if $i\notin S$; such a cocharacter exists because $G$ is simply connected, and the fundamental weights therefore form
a basis for the weight lattice.  The map $\varphi_S:X_{G_S}\to X_G$ is equivariant with respect to the
action of $\Gm$ on $X_G$ given by $\lambda_\rho$.
\end{proof}

\section{Recursive formula}
We begin this section with a review of the statement from \cite{KLS} that will be our main technical
tool, providing a link between point counts over finite fields and intersection cohomology.
Let $$X = \bigsqcup_{x\in Q} V_x$$
be a stratified algebraic variety over $\F_q$, where $Q$ is a finite poset and $x\leq y$
if and only if the stratum $V_x$ is contained in the closure of $V_y$.

Next, suppose that, for each $x\in P$, we have an affine subvariety $C_x\subset X$ that is a
{\bf conical normal slice} to the stratum $V_x$.  
The phrase ``normal slice'' means that $C_x$ intersects
$V_x$ in a single point, and in a neighborhood of that point,
$X$ looks like the product of $V_x$ with $C_x$.  In particular,
we require that, for all $x\leq y\in Q$,
the restriction of the IC sheaf of $\overline V_{\!y}(\bar\F_q)$ to $C_x(\bar\F_q)$ is isomorphic to the IC sheaf of $X_{xy}(\bar\F_q)$.
The word ``conical'' means that we
have a stratification preserving action of $\Gm$ on $X$ (depending on the element $x\in Q$) that takes $C_x$ to itself
and contracts it to the point of intersection with $V_x$.  
For all $x\leq y\in Q$, let $V_{xy} := C_x\cap V_y$, and let $X_{xy} := \overline{V}_{xy}$.
Then, for all $x\leq z\in Q$, we have \cite[Lemma 3.1]{KLS}
$$X_{xz} = \bigsqcup_{x\leq y\leq z} V_{xy}.$$

Finally, suppose that we have a collection of polynomials $\{\kappa_{xy}(t)\mid x\leq y\in Q\}$
such that, for all $x\leq y\in Q$ and all $s\geq 1$, we have
$$\kappa_{xy}(q^s) = |V_{xy}(\F_{q^s})|.$$
Choose a prime $l$ that does not divide $q$, and for all $x\leq y\in Q$, let
$$f_{xy}(t) := \sum_{i\geq 0} t^i \dim \IH^{2i}_{\text{\'et}}(X_{xy}(\bar\F_q); \Q_l)$$
be the even degree $l$-adic \'etale intersection cohomology Poincar\'e polynomial of $X_{xy}(\bar\F_q)$, and let $d_{xy} = \dim X_{xy}$.

\begin{theorem}{\em \cite[Theorem 3.3]{KLS}}\label{KLS-recursion}
For all $x< z\in Q$, the $l$-adic \'etale
intersection cohomology
of $X_{xz}$ vanishes in odd degree, $\deg f_{xz}(t) < d_{xz}/2$, and 
$$t^{d_{xz}}f_{xz}(t^{-1}) = \sum_{x\leq y\leq z}\kappa_{xy}(t)f_{yz}(t).$$
\end{theorem}

\begin{remark}\label{outline}
We give a brief outline of the idea behind the proof of Theorem \ref{KLS-recursion}.
The Grothendieck--Lefschetz trace formula (an $l$-adic \'etale version of the Lefschetz fixed point theorem in topology) tells us that
$$\sum_{j\geq 0}(-1)^j \operatorname{tr}\!\Big(\Fr\actson \IH_{\text{\'et},c}^j(X_{xz}(\bar\F_q);\Q_l)\Big)
= \sum_{p\in X_{xz}(\Fq)}\sum_{j\geq 0}(-1)^j \operatorname{tr}\!\Big(\Fr^j\actson\IH^j_{\text{\'et}}(\IC_{X_{xz}(\bar\F_q),p}(\Q_l))\Big),$$
where $\IH_{\text{\'et},c}$ denotes compactly supported intersection cohomology, $\IC_{X_{xy}(\bar\F_q),p}(\Q_l)$ denotes the stalk of the IC sheaf at the point $p$ with coefficients in $\Q_l$, and $\Fr$ denotes the Frobenius automorphism.
By Poincar\'e duality, $\IH_{\text{\'et},c}^j(X_{xz}(\bar\F_q);\Q_l)$ is isomorphic to $\IH_{\text{\'et}}^{2d_{xz}-j}(X_{xz}(\bar\F_q);\Q_l)$.
For each point $p\in V_{xy}\subset X_{xz}$, we have
$$\IH^*_{\text{\'et}}(\IC_{X_{xz}(\bar\F_q),p}(\Q_l))
\cong \IH^*_{\text{\'et}}(X_{yz}(\bar\F_q);\Q_l),$$
so the right-hand side becomes
$$\sum_{x\leq y\leq z}\kappa_{xy}(q)\sum_{j\geq 0}(-1)^j \operatorname{tr}\!\Big(\Fr^j\actson\IH^j_{\text{\'et}}(X_{yz}(\bar\F_q);\Q_l)\Big).$$

Suppose that we knew that all of these intersection cohomology groups vanished in odd degree and that the Frobenius automorphism
always acts on $\IH^{2i}$ as multiplication by $q^i$.  This would allow us to express our trace formula succinctly as
$$q^{d_{xz}}f_{xz}(q^{-1}) =  \sum_{x\leq y\leq z}\kappa_{xy}(q)f_{yz}(q).$$  Moreover, if this holds with $q$ replaced
by $q^s$ for any $s\geq 1$, then it holds with $q$
replaced by a formal variable $t$.

The difficult part is showing that the intersection cohomology vanishes in odd degree and that the Frobenius acts in the prescribed way,
which is proved using a delicate induction.  This approach was employed for classical Schubert varieties in \cite{KL80}, for toric varieties
in \cite{DL,Fieseler}, for hypertoric varieties in \cite{PW}, and arrangement Schubert varieties in \cite{EPW}.
The main result of \cite{KLS} is a unification of these arguments into a machine that can be used off the shelf.
\end{remark}

We now apply Theorem \ref{KLS-recursion} in our setting.
Let $G$ be a simply connected, split, semisimple group scheme over the integers, and let $X_G$
be its associated base affine space, which is again a scheme over $\Z$.
Lemma \ref{spread} says that the main
technical results of the previous section hold over a finite field
of sufficiently large characteristic.  The lemma follows, for example, from repeated applications of \cite[Theorem 3.2.1(i)]{Poonen}.

\begin{lemma}\label{spread}
If $q$ is a power of a sufficiently large prime, then 
Theorem \ref{gjs-strat} and Lemma \ref{slice} both hold with $\C$    
replaced by the finite field $\F_q$.
\end{lemma}

Let $m_1,\ldots,m_{\rk G}$ be the exponents
of $G$, and let $$f_G(t) := (1-t^{m_1+1})\cdots (1-t^{m_{\rk G}+1}).$$
Let $d_G := \deg f_G(t) = \dim X_G$.
This polynomial will be relevant to us in two different ways.  First, we have (see e.g. \cite[Theorem 9.4.10]{carter} or \cite[\S 24.1]{malle-testerman})
\begin{equation}\label{count}|G(\mathbb{F}_q)| = (-1)^{\rk G} q^{d_G-\rk G}f_G(q)\end{equation} for any prime power $q$.
Second, we have
$$\frac{1}{f_G(t)} = \sum_{i=0}^\infty t^i \dim \HH_{G(\C)}^{2i}(\bullet; \Q),$$
where we write $\bullet$ to denote a point.
Let $$P_G(t) := \sum_{i=0}^\infty t^i \dim \IH^{2i}(X_G(\C); \Q).$$
If we choose a prime $l$ that does not divide $q$, then we may replace the topological intersection cohomology of the space $X_G(\C)$
with the $l$-adic \'etale intersection cohomology of the variety $X_G(\bar\F_q)$ without changing the Poincar\'e polynomial 
\cite[Proposition 10.4.1(i)]{Kirwan-Woolf}.

\begin{theorem}\label{general group}
The intersection cohomology of $X_G(\C)$ vanishes in odd degree.
If $G$ is not the trivial group, then\footnote{When
$G$ is the trivial group, $P_G(t) = 1$ and $d_G = 0 = \deg P_G(t)$, so this
inequality fails.}
$\deg P_G(t) < d_G/2$, and \begin{equation}\label{recurse}\frac{P_G(t)}{f_G(t)} = \sum_{S\subset \Delta}\frac{t^{d_{G_S}}P_{G_S}(t^{-1})}{f_{G_S}(t)}.\end{equation}
\end{theorem}

\begin{remark}
Before proving Theorem \ref{general group}, we explain why it allows us to
compute $P_G(t)$ recursively.  There is a slight subtlety here, since
$G_\emptyset = G$ and therefore $P_G(t)$ appears on both sides of the 
equation.  Let us rewrite this equation as
\begin{equation}\label{difference}P_G(t) - t^{d_{G}}P_{G}(t^{-1}) = \sum_{\emptyset \neq S\subset \Delta}\frac{f_G(t)\,t^{d_{G_S}}P_{G_S}(t^{-1})}{f_{G_S}(t)}.\end{equation}
Since $\deg P_G(t) < d_G/2$, the polynomial $t^{d_{G}}P_{G}(t^{-1})$
vanishes in degree less than or equal to $d_G/2$, and therefore
$P_G(t)$ can be obtained from the right hand side of the equation by truncation.
\end{remark}

\begin{proof}
Choose a prime power $q$ as in Lemma \ref{spread}.
Theorem \ref{gjs-strat} and Lemma \ref{spread} tell us that the $G$-orbits in $X_G$ are indexed by subsets $S\subset\Delta$, and the number of $\mathbb{F}_q$-points
on the orbit $O_S$ is equal to
$$|G(\mathbb{F}_q)/H_S(\mathbb{F}_q)|
= \frac{|G(\mathbb{F}_q)|}{|G_S(\mathbb{F}_q)|\cdot |U_S(\mathbb{F}_q)|}.$$
Using Equation \eqref{count} and the fact that 
$\dim U_S = (d_G - \rk G) - (d_{G_S}-\rk G_S)$, this is equal to
$$\frac{(-1)^{\rk G}f_G(q)}{(-1)^{\rk G_S} f_{G_S}(q)}.$$
Lemmas \ref{slice} and \ref{spread} tell us that $O_S$ has a normal slice isomorphic to $X_{G_S}$,
and Lemma \ref{slice actions} says that this slice is conical.

We will now apply Theorem \ref{KLS-recursion} to the variety
$X_G$, which is stratified by the Boolean poset of subsets
of $\Delta$.  We assume that $G$ is nontrivial, so that
$\emptyset$ is strictly contained in $\Delta$.
We have $(X_G)_{\emptyset\Delta} = X_G$,
and for every subset $S\subset \Delta$, we have
$V_{\emptyset S} = V_S = O_S$ and $X_{S\Delta} = C_S \cong X_{G_S}$.
More generally, for any $S\subset T\subset \Delta$,
$V_{ST}$ is isomorphic to an orbit of $G_S$ on $X_{G_S}$, and
therefore has polynomial point count.  Theorem \ref{KLS-recursion}
then tells us that the intersection cohomology of $X_G$
vanishes in odd degree, $\dim P_G(t) < d_G/2$, and
\begin{equation}\label{before simplifying}t^{d_G}P_G(t^{-1}) = \sum_{S\subset \Delta} 
\frac{(-1)^{\rk G}f_G(t)}{(-1)^{\rk G_S}f_{G_S}(t)}
P_{G_S}(t).\end{equation}
Equation \eqref{recurse} follows from this formula by replacing $t$
with $t^{-1}$ and making use of the identity
$(-1)^{\rk G} f_G(t^{-1}) = t^{-\rk G} f_G(t)$.
\end{proof}

\begin{remark}\label{categorify}
We now explain how to categorify Equation \eqref{recurse}.
Let $\IC_{X_G}$ be the intersection cohomology sheaf on $X_G(\C)$.  For each subset $S\subset\Delta$, let $\iota_S:O_S(\C)\hookrightarrow X_G(\C)$ be the inclusion of the corresponding orbit.  
The stratification of $X_G(\C)$ by orbits induces a filtration by supports on the complex of global sections of an injective resolution of $\IC_{X_G}$.  This filtered complex gives rise to a spectral sequence with
$$E_1^{p,q} = \bigoplus_{\codim O_S = p}\mathbb{H}^{p+q}_{G(\C)}(\iota_S^!\IC_{X_G}),$$
converging to $\IH^*_{G(\C)}(X_G(\C); \Q)$ \cite[Section 3.4]{BGS96}.
The hypercohomology of the complex $\iota_S^!\IC_{X_G}$ is a $G(\C)$-equivariant local system
on $O_S(\C)$ whose fiber at the point $v_S$ is equal to $\IH_c^*(X_{G_S}(\C); \Q)$, the compactly supported intersection cohomology of the normal slice.  Since the stabilizer $H_S(\C)$ is connected,
this local system is canonically trivial, and we therefore have
\beq \mathbb{H}^*_{G(\C)}(\iota_S^!\IC_{X_G}) &\cong& \HH^*_{G(\C)}(O_S(\C); \Q)\otimes \IH_c^*(X_{G_S}(\C); \Q)\\
&\cong& \HH^*_{H_S(\C)}(\bullet;\Q)\otimes \IH_c^*(X_{G_S}(\C); \Q)\\
&\cong& \HH^*_{G_S(\C)}(\bullet;\Q)\otimes \IH_c^*(X_{G_S}(\C); \Q),\eeq
where the last isomorphism follows from the fact that the inclusion of $G_S(\C)$ into $H_S(\C)$ is a homotopy
equivalence.
Returning to our spectral sequence, this means that
$$E_1^{p,q}\; \cong \bigoplus_{\substack{\codim O_S = p\\ j+k=p+q}} H^j_{G_S(\C)}(\bullet; \Q)\otimes \IH_c^k(X_{G_S}(\C); \Q).$$
The left-hand side of Equation \eqref{recurse} is the Hilbert series of the $E_\infty$-page (with respect to the total grading)
and the right-hand side is the Hilbert series of the $E_1$-page.  Thus Equation \eqref{recurse} implies
that the spectral sequence degenerates at the $E_1$-page.  In other words, taking the associated graded with respect
to the filtration of $\IH^*_{G(\C)}(X_G(\C); \Q)$ induced by the orbit stratification, we have
$$\operatorname{gr} \IH^*_{G(\C)}(X_G(\C); \Q) \cong E_\infty = E_1 = \bigoplus_S
\HH_{G_S(\C)}^*(\bullet;\Q)\otimes \IH^*_c(X_{G_S}(\C);\Q),$$ which categorifies Equation \eqref{recurse}.
\end{remark}

\begin{remark}  Turning Remark \ref{categorify} into an alternate proof of Equation \eqref{recurse} would require showing independently
that the intersection cohomology vanishes in odd degree and that the spectral sequence degenerates, 
which would require a calculation of mixed Hodge weights.  This could be achieved via an inductive argument
similar in flavor to the one referenced in Remark \ref{outline}.
\end{remark}

Section \ref{sec:A} will be devoted to understanding Equation \ref{recurse} in the case where $G=\SL_n$.
However, we will conclude this section by considering the case where $G = G_2$.

\begin{example}\label{G2}
Consider the exceptional group $G_2$, and let 
$$P_{G_2}(t) := \sum_{i\geq 0} t^i \dim \IH^{2i}(X_{G_2}(\C); \Q).$$  
We have $\Delta = \{1, 2\}$, where $\varpi_1$ is short and $\varpi_2$ is long.
The space $X_{G_2}$ has four orbits: the dense orbit $G_2/U$, the two intermediate orbits $O_{\{1\}}$ and $O_{\{2\}}$,
and the fixed point $\{v_\emptyset\}$.  
We have $G_{\{1\}}\cong G_{\{2\}}\cong \SL_2$, hence the two intermediate orbits have normal slices
isomorphic to $X_{\SL_2} \cong \mathbb{A}^2$.  We have $f_{G_2}(t) = (1-t^2)(1-t^6)$ \cite{Dickson01} and $f_{\SL_2}(t) = 1-t^2$,
so Equation \eqref{difference} says that
$$P_{G_2}(t) - t^8P_{G_2}(t^{-1}) = (1-t^2)(1-t^6)\left(2\frac{t^2}{1-t^2} + 1\right)
= 1 + t^2 - t^6-t^8.$$
Since the degree of $P_{G_2}(t)$ is strictly less than $8/2 = 4$, this implies that
$P_{G_2}(t) = 1 + t^2$.
\end{example}

\section{Type A}\label{sec:A}
In this section, we interpret Theorem \ref{general group} in the special case where $G=\SL_n$.
As in Example \ref{Sln2}, we write $X_n := X_{\SL_n}$.
Similarly, we write $d_n := d_{\SL_n} = \binom{n+1}{2}-1$, $P_n(t) := P_{\SL_n}(t)$, and
$$f_n(t) := f_{\SL_n}(t) = (1-t^2)\cdots (1-t^n).$$
For any composition $\sigma = (\sigma_1,\ldots,\sigma_r)$ of $n$,
we have $G_\sigma \cong \SL_{\sigma_1}\times\cdots\times\SL_{\sigma_r}$ and
$X_\sigma \cong X_{\sigma_1}\times\cdots\times X_{\sigma_r}$, so Equation \eqref{recurse} may be translated
as follows:
\begin{equation}\label{sln-recurse}\frac{P_n(t)}{f_n(t)} = \sum_r \sum_{\sigma_1+\cdots+\sigma_r=n}\prod_{i=1}^r \frac{t^{d_{\sigma_i}} P_{\sigma_i}(t^{-1})}{f_{\sigma_i}(t)}.\end{equation}
This recursion can be reformulated in terms of a sum with a simpler index set.

\begin{corollary}\label{second recursion}
For all positive integers $n$, we have
$$P_n(t) - t^{d_n}P_n(t^{-1}) = \sum_{s=1}^{n-1} \frac{f_n(t)}{ f_s(t)f_{n-s}(t)} t^{d_s}P_s(t^{-1}) P_{n-s}(t).$$
\end{corollary}

\begin{proof}
By Equation \eqref{sln-recurse}, we have
\begin{eqnarray*}P_n(t) 
&=& f_n(t)\sum_{r=1}^n\sum_{\sigma_1+\cdots+\sigma_r=n}\prod_{i=1}^r \frac{t^{d_{\sigma_i}}P_{\sigma_i}(t^{-1})}{f_{\sigma_i}(t)}\\
&=& P_n(t) + f_n(t)\sum_{r=2}^n\sum_{\sigma_1+\cdots+\sigma_r=n} \prod_{i=1}^r \frac{t^{\sigma_i}P_{\sigma_i}(t^{-1})}{f_{\sigma_i}(t)}\\
&=& P_n(t) + f_n(t)\sum_{\sigma_r=1}^{n-1}\frac{t^{\sigma_r}P_{\sigma_r}(t^{-1})}{f_{\sigma_r}(t)}\sum_{r= 2}^n\sum_{\sigma_1+\cdots+\sigma_{r-1}=n-\sigma_r}\prod_{i=1}^{r-1}\frac{t^{d_{\sigma_i}}P_{\sigma_i}(t^{-1})}{f_{\sigma_i}(t)}\\
&=& P_n(t) + f_n(t)\sum_{\sigma_r=1}^{n-1} \frac{t^{\sigma_r}P_{\sigma_r}(t^{-1})}{f_{\sigma_r}(t)}\cdot\frac{P_{n-\sigma_r}(t)}{ f_{n-\sigma_r}(t)}.
\end{eqnarray*}
Letting $s=\sigma_r$ gives the desired formula.
\end{proof}

\begin{example}\label{SL2}
When $n=2$, Corollary \ref{second recursion} tells us that
$$P_2(t) - t^2 P_2(t^{-1}) = 1 - t^2.$$
Since we know that the degree of $P_2(t)$ is strictly less than $d_2/2 = 1$, this implies that $P_2(t) = 1$.
This is consistent with the fact that $X_2\cong\mathbb{A}^2$.
\end{example}

\begin{example}\label{SL3}
When $n=3$, Corollary \ref{second recursion} tells us that
$$P_3(t) - t^5P_3(t^{-1}) = \frac{f_3(t)}{f_1(t)f_2(t)}P_1(t^{-1})P_2(t) + \frac{t^2f_3(t)}{f_2(t)f_1(t)}P_2(t^{-1})P_1(t)  = 1 + t^2 - t^3 - t^5.$$
As we know that the degree of $P_3(t)$ is strictly less than $d_3/2 = 5/2$, this implies that we have $P_3(t) = 1 + t^2$.
This agrees with the calculation of $P_3(t)$ in \cite{HJ}.
\end{example}

\begin{example}\label{SL4}
When $n=4$, Corollary \ref{second recursion} tells us that
\beq P_4(t) - t^{9}P_4(t^{-1}) &=& 
\frac{f_4(t)}{f_1(t)f_3(t)}P_1(t^{-1})P_3(t) + \frac{t^2 f_4(t)}{f_2(t)f_2(t)}P_2(t^{-1})P_2(t) + \frac{t^5 f_4(t)}{f_3(t)f_1(t)} P_3(t^{-1})P_1(t)\\
&=& (1-t^4)(1+t^2) + t^2 (1-t^3)(1+t^2) + t^5(1-t^4)(1+t^{-2})\\
&=& 1+2t^2 + t^3 - t^6 - 2t^7-t^9.
\eeq
Since we know that the degree of $P_4(t)$ is strictly less than $d_4/2 = 9/2$, this implies that we have $P_4(t) = 1 + 2t^2 + t^3$.
See the appendix for calculations of $P_n(t)$ up to $n=13$.
\end{example}

Consider the generating series
$$\Psi(t,u) := 1 + \sum_{n=1}^\infty u^n \frac{P_n(t)}{f_n(t)}.$$
The recursion for $P_n(t)$ in Corollary \ref{second recursion} can be translated into a function equation
for $\Psi(t,u)$ as follows.

\begin{proposition}\label{functional}
We have
$\Psi(t^{-1}, u)\Psi(t,-u) = 1.$
\end{proposition}

\begin{proof}
Using the fact that $(-1)^{n-1}f_n(t^{-1}) = t^{-d_n}f_n(t)$, we have
\beq
\Big(\Psi(t^{-1},u) - 1\Big)\Big(\Psi(t, -u) - 1\Big) &=& \sum_{m=1}^\infty u^m \frac{P_m(t^{-1})}{f_m(t^{-1})} \cdot\sum_{n=1}^{\infty}(-u)^n\frac{P_n(t)}{f_n(t)}\\
&=& - \sum_{m=1}^\infty (-u)^m \frac{t^{d_m}P_m(t^{-1})}{f_m(t)}\cdot \sum_{n=1}^{\infty}(-u)^n\frac{P_n(t)}{f_n(t)}\\
&=& - \sum_{k=2}^\infty \frac{(-u)^k}{f_k(t)}\sum_{m=1}^{k-1} \frac{f_k(t)}{f_{m}(t)f_{k-m}(t)}t^{d_m}P_m(t^{-1})P_{k-m}(t),
\eeq
where the last line is obtained by putting $k=m+n$.  Corollary \ref{second recursion} says that the internal sum is equal to $P_k(t) - t^{d_k}P_k(t^{-1})$,
thus we have
\beq
\Big(\Psi(t^{-1},u) - 1\Big)\Big(\Psi(t, -u) - 1\Big) &=&
\sum_{k=2}^\infty \frac{(-u)^k}{f_k(t)}t^{d_k}P_k(t^{-1})
- \sum_{k=2}^\infty \frac{(-u)^k}{f_k(t)}P_k(t)\\
&=&-\sum_{k=2}^\infty \frac{u^k}{f_k(t^{-1})}P_k(t^{-1}) - \sum_{k=2}^\infty \frac{(-u)^k}{f_k(t)}P_k(t)\\
&=&\Big(1+u-\Psi(t^{-1},u)\Big)+\Big(1-u-\Psi(t,-u)\Big)\\
&=& 2 - \Psi(t^{-1},u)- \Psi(t,-u).
\eeq
Adding $\Psi(t^{-1},u)+\Psi(t,-u) - 1$ to both sides gives the desired equation.
\end{proof}

\begin{remark}\label{unique}
The power series $\Psi(t,u)$ is uniquely characterized by Proposition \ref{functional}
along with the fact that the coefficient of $u$ is equal to 1 and, for all $n\geq 2$,
$f_n(t)$ times the coefficient of $u^n$ is equal to a polynomial in $t$ of degree strictly less than $d_n/2$.
\end{remark}

\begin{remark}\label{Balazs}
The {\bf plethystic logarithm} $$\operatorname{PLog} \Psi(t,u) = \sum_{i,n\geq 0}e(i,n)t^iu^n$$ is the power series with integer coefficients uniquely determined
by the equation $$\Psi(t,u) = \prod_{n,i\geq 0}\frac{1}{(1-t^iu^n)^{e(i,n)}}.$$
Numerical evidence strongly suggests that these coefficients are non-negative, and furthermore that there
exist polynomials $Q_n(t)$ with non-negative integer coefficients  such that
$$\operatorname{PLog} \Psi(t,u) = u + t^2 \sum_{n=2}^\infty \frac{u^n Q_n(t)}{f_n(t)}.$$
We thank Vladimir Dotsenko and Bal\'azs Szendr\H{o}i for conjecturing the non-negativity of the coefficients $e(i,n)$, Max Alekseyev
for checking this conjecture for all $i,n\leq 50$, and Anton Mellit and Hjalmar Rosengren for making the stronger conjecture, which was verified by Szendr\H{o}i up to $n=7$.
We also thank MathOverflow \cite{MOF} for providing
a forum for this discussion.
\end{remark}

\section{The Betti numbers}
In this section we study the coefficient of $t^i$ in $P_n(t)$ as a function of $n$.
Let
$$P_n(t) = \sum_{i\geq 0} c_i(n) t^i.$$
In order to make use of the recursion in Corollary \ref{second recursion}, it will also be useful to write
$$\frac{f_n(t)}{f_s(t)f_{n-s}(t)} = \sum_{i\geq 0} b_{i,s}(n) t^i.$$

\begin{lemma}\label{Gaussian}
For all $i\geq 0$ and $s\geq 1$, the function $b_{i,s}(n)$ is constant for $n\geq \max(i,1)+s$.
\end{lemma}

\begin{proof}
We treat the $i=0$ and $i=1$ cases separately.
We have $\frac{f_n(0)}{f_s(0)f_{n-s}(0)} = 1$, so $b_{0,s}(n) = 1$ for all $n\geq 1+s$.  
Using the Taylor series expansion of $\frac{1}{1-t^r}$,
we can see that the coefficient
of $t$ in $\frac{f_n(t)}{f_s(t)f_{n-s}(t)}$ is 0, so $b_{1,s}(n) = 0$ for all $0<s<n$.

We now proceed by induction on the quantity $i+s$.  We have
$$\frac{f_n(q)}{f_s(q)f_{n-s}(q)} = (1-q)\binom{n}{s}_{\!\!q},$$
and the well-known Gaussian binomial coefficient identity
$$\binom{n}{s}_{\!\!q} = q^s\binom{n-1}{s}_{\!\!q} + \binom{n-1}{s-1}_{\!\!q}$$ translates to the identity
$$\frac{f_n(t)}{f_s(t)f_{n-s}(t)} = t^s \frac{f_{n-1}(t)}{f_s(t)f_{n-1-s}(t)} + \frac{f_{n-1}(t)}{f_{s-1}(t)f_{n-s}(t)}.$$ 
Taking coefficients of $t^i$, this means that
$$b_{i,s}(n) = b_{i-s,s}(n-1) + b_{i,s-1}(n-1).$$
Since we have already treated the cases when $i=0$ and $i=1$, we may assume that $i\geq 2$.
This means that our inductive hypothesis implies that the right-hand side is constant for
$n-1\geq i+s-1$, or equivalently for $n\geq i+s$.
\end{proof}

\begin{proposition}\label{growth}
Fix a non-negative integer $i$.  The function $c_i(n)$ is given by a polynomial in $n$ of degree at most $i/2$ for all $n\geq i$.
\end{proposition}

\begin{proof}
We proceed by induction on $i$.  The base case is $i=0$, which holds because $c_0(n) = 1$ for all $n$.
Now fix $i>0$ and assume that the statement holds
for all $j<i$.  Our strategy will be to prove that
the quantity $c_i(n)-c_i(n-1)$ is given by a polynomial
in $n$ for all $n>i$, which implies that $c_i(n)$
is given by a polynomial in $n$ for all $n\geq i$.

Let $n>i$ be given.
By Corollary \ref{second recursion}, we have
$$c_i(n) - c_{d_n-i}(n) = \sum_{s=1}^{n-1}\sum_{\substack{p+q+r=i\\ p,q,r\geq 0}} b_{p,s}(n)c_{d_s-q}(s)c_r(n-s).$$
Since $n>i>0$, we have $c_{d_n-i}(n)=0$.
If $s>i\geq q$, then $c_{d_s-q}(s) = 0$, so we may restrict the upper bound of the outer sum to $i$, which is independent of $n$.  
Thus we have
\begin{equation}\label{ci-recursion}c_i(n) = \sum_{s=1}^{i}\sum_{\substack{p+q+r=i\\ p,q,r\geq 0}} b_{p,s}(n)c_{d_s-q}(s)c_r(n-s).\end{equation}

By Lemma \ref{Gaussian}, $b_{p,s}(n)$ is constant for $n\geq \max(p,1)+s$.
The condition that $n\geq 1+s$ is automatic from the fact that $s\leq i<n$.
If $n<p+s$, then we have $p+q+r=i<n<p+s$ with both inequalities strict.  This means that $s-q\geq 2$, which in turn implies that $c_{d_s-q}(s) = 0$.
Thus every term in our sum is equal to a constant times $c_r(n-s)$.

When $r=i$, we necessarily have $p=q=0$.  The only nonzero term of this form occurs when $s=1$, and we obtain $c_i(n-1)$.
When $r=i-1$, we either have $p=1$ and $q=0$, in which case $b_{1,s}(n)=0$ for all $s$, or $p=0$ and $q=1$, in which case $c_{d_s-1}(s) = 0$ for all $s$.
When $r\leq i-2$, our inductive hypothesis implies that $c_r(n-s)$ is given by a polynomial in $n$ of degree at most $r/2 \leq i/2 - 1$ whenever $n-s\geq r$.
We know that $n > i = p+q+r$, and therefore $n-s >p+q+r-s$.  We also know that $c_{d_s-q}(s) = 0$ unless $s\leq q+1$.  Thus, for every nonzero term,
$n-s>p+q+r-s\geq p+r-1\geq r-1$, and therefore $n-s\geq r$.

We have now shown that $c_i(n)$ is equal to $c_i(n-1)$ plus a function of $n$ that agrees with a polynomial of degree at most $i/2 - 1$ whenever $n>i$.
This implies that $c_i(n)$ is given by a polynomial in $n$ of degree at most $i/2$ whenever $n\geq i$.
\end{proof}

Since the function $c_i(n)$ takes integer values,
Proposition \ref{growth} implies that there exist integers $a_{i,k}$ for $0\leq k\leq i/2$ with the property that, for all $n\geq i$,
$$c_i(n) = \sum_{k=0}^{\lfloor i/2\rfloor} a_{i,k}\binom{n-i}{k}.$$

\begin{conjecture}\label{binomial}
The integers $a_{i,k}$ are all non-negative.
\end{conjecture}

Conjecture \ref{binomial} is motivated by the following corollary of Proposition \ref{growth}, in which we compute the coefficients
$a_{i,k}$ for all $i\leq 9$.

\begin{corollary}\label{terms of small degree}
We have the following identities:
\beq
\text{for all $n\geq 0$, }\quad c_0(n) &=& 1\\
\text{for all $n\geq 1$, }\quad c_1(n) &=& 0\\
\text{for all $n\geq 2$, }\quad c_2(n) &=& n-2\\
\text{for all $n\geq 3$, }\quad c_3(n) &=& n-3\\
\text{for all $n\geq 4$, }\quad c_4(n) 
&=& \binom{n-4}{2} + 2\binom{n-4}{1}\\
\text{for all $n\geq 5$, }\quad c_5(n) 
&=&  2\binom{n-5}{2} + 4\binom{n-5}{1} + \binom{n-5}{0}\\
\text{for all $n\geq 6$, }\quad c_6(n) 
&=& \binom{n-6}{3} + 5\binom{n-6}{2} + 9\binom{n-6}{1} + 6\binom{n-6}{0}\\
\text{for all $n\geq 7$, }\quad c_7(n) 
&=& 3\binom{n-7}{3} + 12\binom{n-7}{2} + 20\binom{n-7}{1} + 15\binom{n-7}{0}\\
\text{for all $n\geq 8$, }\quad c_8(n) 
&=& \binom{n-8}{4} + 9\binom{n-8}{3} + 30\binom{n-8}{2} + 53\binom{n-8}{1} + 50\binom{n-8}{0}\\
\text{for all $n\geq 9$, }\quad c_9(n) 
&=& 4\binom{n-9}{4} + 25\binom{n-9}{3} + 73\binom{n-9}{2} + 125\binom{n-9}{1} + 123\binom{n-9}{0}.
\eeq
\end{corollary}

\begin{proof}
By Proposition \ref{growth}, we can compute $c_i(n)$ for $n\geq i$ by polynomial interpolation as $n$ ranges from $i$ to $\lfloor 3i/2\rfloor$.
We do this using the formulas in the appendix.
\end{proof}

\begin{example}\label{c2}
Corollary \ref{terms of small degree} says that, for all $n\geq 2$, $\dim \IH^4(X_n(\C); \Q) = c_2(n) = n-2$, or equivalently 
$\dim \IH^4_{\SL_n(\C)}(X_n(\C); \Q) = n-1$.  Indeed, $\dim \IH^4_{\SL_n(\C)}(X_n(\C); \Q)$ is equal to the coefficient of $t^2$ in the power series
$P_n(t)/f_n(t)$; Equation \eqref{sln-recurse} shows us that this coefficient is $n-1$, corresponding to the $n-1$ different compositions consisting
of a two and a bunch of ones.  At the level of vector spaces, Remark \ref{categorify} tells us that $\IH^4_{\SL_n}(X_n(\C);\Q)$ decomposes as
a direct sum of $n-1$ copies of the 1-dimensional vector space $H_{\SL_2}^0(\bullet; \Q)\otimes \IH^4_c(X_{\SL_2}(\C);\Q)$, one for each such composition.\footnote{Since
all of these terms appear in the same entry of the $E_1$-page, the filtration with respect to which we need to take the associated graded is trivial.}  Similarly,
$\dim \IH^6_{\SL_n(\C)}(X_n(\C); \Q) = n-2$, and $\IH^6_{\SL_n}(X_n(\C);\Q)$ decomposes as
a direct sum of $n-2$ copies of the 1-dimensional vector space $H_{\SL_3}^0(\bullet; \Q)\otimes \IH^6_c(X_{\SL_3}(\C);\Q)$, one for each composition consisting
of a three and a bunch of ones.
\end{example}    

\begin{proposition}
If $i$ is even, then $a_{i,i/2}=1$.
If $i$ is odd, then $a_{i,(i-1)/2} = (i-1)/2$.
\end{proposition}

\begin{proof}
If $i$ is even, then Equation \eqref{ci-recursion}
gives us the difference equation
$$c_i(n)-c_i(n-1) = c_{i-2}(n-2) + O((i-4)/2).$$
If $i$ is odd, we get
$$c_i(n)-c_i(n-1) = c_{i-2}(n-2) + c_{i-3}(n-3) + O((i-5)/2).$$
The proposition now follows from induction on $i$.
\end{proof}

\begin{remark}\label{Stirling}
The formulas for $c_2(n)$ and $c_3(n)$ are related to identities involving Stirling numbers. More precisely, to prove directly that $c_2(n)=n-2$, 
assume inductively that $$P_m(t)=1 +(m-2)t^2 + \cdots$$ for all $m<n$. We observe that
\[
\frac{f_n(t)}{f_{\sigma_1}(t) \ldots f_{\sigma_r}(t)} = 1 +(k_{\sigma}-1)t^2 + \ldots,
\]
where $k_{\sigma}$ is the number of $i$ such that $\sigma_i>1$.  The inductive step is equivalent to the identity
\[
n-2 = \sum_{r=2}^{n} (-1)^{r}(n-r-1) \binom{n-1}{r-1},
\]
which in turn is equivalent to the statement
\[
\sum_{k=0}^{n-1} (-1)^k k \binom{n-1}{k}=0.
\]
The left-hand side is equal to
$(-1)^{n-1} (n-1)!$ times the Stirling number of the second kind $S(1,n-1)$ \cite[Equation 1.94a]{St}, which counts the number of ways of partitioning a 1-element set into $n-1$ nonempty subsets, and therefore vanishes for $n \geq 3$.

To prove directly that $c_3(n)=n-3$, a similar induction, along with our calculation of $c_2(n)$, allow us to reduce to the identity
\begin{equation}\label{c3}\sum_{r=1}^n(-1)^r\sum_{\sigma_1+\cdots+\sigma_r=n}(r-k_\sigma)=0.\end{equation}
We are grateful to Paul Balister for explaining the following proof of Equation \eqref{c3}.
The internal sum can be rewritten as a sum over $i$ of the number of compositions $\sigma$ into $r$ parts with $\sigma_i=1$.
This in turn is equal to the sum over $k$ of the number of compositions of $k$ into $i-1$ parts times the number of compositions of $n-k-1$ into $r-i$ parts.
Letting $$A_k := \sum_{r=1}^k(-1)^r\sum_{\sigma_1+\cdots+\sigma_r=n}1,$$ we can therefore express the left-hand side of Equation \eqref{c3} as 
$-\sum_{i} A_{i-1} A_{n-i}$.  But $A_k=0$ for $k\geq 2$, so this sum vanishes for $n\geq 4$.
\end{remark}

\section{Two bigraded rings}\label{sec:rings}
In this section, $X_n$ will always mean $X_n(\C)$, and intersection homology and cohomology will always be taken with rational coefficients.
Let $m$ and $n$ be positive integers, and consider the normal slice map
$$\varphi_{m,n}:X_m\times X_n \to X_{m+n}$$ from Example \ref{Sln2} with $r=2$.
This is a normally nonsingular inclusion, and therefore induces a graded map
$$(\varphi_{m,n})_*:\IH_*(X_m)\otimes \IH_*(X_n) \to \IH_*(X_{m+n})$$
on intersection homology.  Given a third positive integer $l$, Example \ref{Sln2} shows that
$$\varphi_{l,m+n}\circ(\id_{X_l}\times\varphi_{m,n}) = \varphi_{l,m,n} = \varphi_{l+m,n}\circ(\varphi_{l,m}\times\id_{X_n}),$$
therefore our two natural maps
$$\IH_*(X_l)\otimes \IH_*(X_m)\otimes \IH_*(X_n) \to \IH_*(X_{l+m+n})$$
agree.  This in turn means that our maps define an associative, bigraded ring structure on the vector space
$$R := \Q \oplus \bigoplus_{n=1}^\infty \IH_*(X_n).$$

We can also build a version of this ring using equivariant intersection homology goups 
$$\IH_{2i}^{\SL_n}(X_n) := \IH^{2i}_{\SL_n}(X_n)^*.$$
For any $m$ and $n$, we have maps
$$\IH_*^{\SL_m}(X_m)\otimes \IH_*^{\SL_n}(X_n)\cong \IH_*^{\SL_m\times\SL_n}(X_m\times X_n)
\to \IH_*^{\SL_{m} \times\SL_{n}}(X_{m+n}) \to \IH_*^{\SL_{m+n}}(X_{m+n}),$$
where the first map is induced by the $(\SL_m\times\SL_n)$-equivariant
normally nonsingular inclusion $\varphi_{m,n}$ and the second
is induced by the inclusion of $\SL_m\times\SL_n$ into $\SL_{m+n}$.
We thus obtain an associative bigraded ring structure on
$$\widehat R := \Q \oplus \bigoplus_{n=1}^\infty \IH_*^{\SL_n}(X_n).$$

\begin{remark}
One motivation for studying the ring $\widehat R$ is that its Hilbert series is equal to the power series $\Psi(t,u)$ from Proposition \ref{functional}.
If the ring $\widehat{R}$ were commutative, then the non-negativity of the plethystic logarithm discussed in Remark \ref{Balazs} would be equivalent
to the statement that $\widehat{R}$ is an infinite polynomial ring with $e(i,n)$ generators in bidegree $(i,n)$.  However, this ring is in fact very far from
being commutative, as we can see below.
\end{remark}

While we are unable to give a complete description of either $R$ or $\widehat{R}$, it is possible to describe these rings in low homological degree.
Let $x\in \IH_0^{\SL_1}(X_1) = \IH_0(\bullet)$, $y\in \IH_4^{\SL_2}(X_2)\cong \HH_4^{\SL_2}(\bullet)$, and $z\in \IH_6^{\SL_3}(X_3)\cong\HH_6^{\SL_3}(\bullet)$ be generators
of their respective 1-dimensional vector spaces.  These classes freely generate the ring $\widehat{R}$ in homological degree at most 6:
\begin{itemize}
\item For all $n\geq 1$, $\IH_0^{\SL_n}(X_n)$ is 1-dimensional with generator $x^n$.
\item For all $n\geq 1$, $\IH_2^{\SL_n}(X_n) = 0$.
\item For all $n\geq 1$, $\IH_4^{\SL_n}(X_n)$ is $(n-1)$-dimensional with basis $\{x^{i-1}yx^{n-1-i}\mid 1\leq i\leq n-1\}$, corresponding to the decomposition described in Example \ref{c2}.
\item For all $n\geq 2$, $\IH_6^{\SL_n}(X_n)$ is $(n-2)$-dimensional with basis $\{x^{i-1}zx^{n-2-i}\mid 1\leq i\leq n-2\}$, also corresponding to the decomposition described in Example \ref{c2}.
\end{itemize}
The subring $R\subset\widehat{R}$ contains the classes $x$, $xy-yx$, and $xz-zx$, and is freely generated by these classes in homological degree at most 6.

\appendix
\section{Appendix}
Here we list the polynomials $P_n(t)$ for $n\leq 13$.
\beq
P_1(t) &=& 1\\
P_2(t) &=& 1\\
P_3(t) &=& 1 + t^2\\
P_4(t) &=& 1+ 2t^2 + t^3\\
P_5(t) &=& 1 + 3t^2 + 2t^3 + 2t^4 + t^5 + 2t^6\\
P_6(t) &=& 1 + 4t^2 + 3t^3 + 5t^4 + 5t^5 + 6t^6 + 5t^7 + 4t^8 + t^9\\
P_7(t) &=& 1 + 5t^2 + 4t^3 + 9t^4 + 11t^5 + 15t^6 + 15t^7 + 20t^8 + 13t^9  + 12t^{10} + 9t^{11} + 9t^{12} + t^{13}\\
P_8(t) &=& 1 + 6t^{2} + 5t^{3} + 14t^{4} + 19t^{5} + 29t^{6} + 35t^{7} + 50t^{8}+ 51t^{9} + 55t^{10} + 55t^{11} + 58t^{12} + 43t^{13}\\
&& +\; 38t^{14} + 30t^{15} + 16t^{16} + 5t^{17}\\
P_9(t) &=& 1 + 7t^{2} + 6t^{3} + 20t^{4} + 29t^{5} + 49t^{6} + 67t^{7} + 103t^{8} + 123t^{9} + 160t^{10} + 178t^{11} + 213t^{12}\\
&& +\; 212t^{13} + 229t^{14} + 215t^{15} + 202t^{16} + 162t^{17} + 137t^{18} + 109t^{19} + 83t^{20} + 35t^{21}\\
P_{10}(t) &=& 1 + 8t^{2} + 7t^{3} + 27t^{4} + 41t^{5} + 76t^{6} + 114t^{7} + 186t^{8} + 248t^{9} + 354t^{10} + 445t^{11} + 569t^{12}\\
&& +\; 666t^{13} + 797t^{14} + 867t^{15} + 944t^{16} + 968t^{17}  + 972t^{18} + 938t^{19} + 888t^{20} + 767t^{21}\\
&& +\; 624t^{22} + 539t^{23} + 420t^{24} + 277t^{25} + 138t^{26}\\
P_{11}(t) &=& 1 + 9t^{2} + 8t^{3} + 35t^{4} + 55t^{5} + 111t^{6} + 179t^{7} + 308t^{8} + 446t^{9} + 683t^{10} + 931t^{11}\\ 
&&  +\; 1284t^{12} + 1639t^{13} + 2131t^{14} + 2554t^{15} + 3068t^{16} + 3516t^{17} + 3978t^{18} + 4299t^{19}\\
&&  +\; 4620t^{20} + 4722t^{21} + 4738t^{22} + 4655t^{23} + 4443t^{24} + 4047t^{25} + 3552t^{26} + 2937t^{27}\\
&& +\; 2514t^{28} + 2029t^{29} + 1484t^{30} + 873t^{31} + 265t^{32}\\
P_{12}(t) &=& 1 + 10t^{2} + 9t^{3} + 44t^{4} + 71t^{5} + 155t^{6} + 265t^{7} + 479t^{8} + 742t^{9} + 1202t^{10}  + 1749t^{11}\\
&& +\; 2561t^{12} + 3511t^{13} + 4828t^{14} + 6255t^{15} + 8049t^{16} + 9969t^{17} + 12172t^{18} + 14362t^{19}\\
&& +\; 16721t^{20} + 18888t^{21} + 20965t^{22} + 22755t^{23} + 24178t^{24} + 25133t^{25} + 25498t^{26}\\
&& +\; 25195t^{27} + 24670t^{28} + 23456t^{29} + 21772t^{30} + 19414t^{31} + 16711t^{32} + 14123t^{33}\\
&& +\; 12023t^{34} + 9482t^{35} + 6833t^{36} + 4006t^{37} + 1317t^{38}\\
P_{13}(t) &=& 1 + 11t^{2} + 10t^{3} + 54t^{4} + 89t^{5} + 209t^{6} + 375t^{7} + 710t^{8} + 1165t^{9} + 1980t^{10} + 3043t^{11}\\
&& +\; 4692t^{12} + 6807t^{13} + 9838t^{14} + 13505t^{15} + 18404t^{16} + 24159t^{17} + 31296t^{18} + 39361t^{19}\\
&& +\; 48823t^{20} + 58981t^{21} + 70278t^{22} + 81886t^{23} + 93869t^{24} + 105612t^{25} + 116901t^{26}\\
&& +\; 126688t^{27} + 135618t^{28} + 142267t^{29} + 147027t^{30} + 148755t^{31} + 147909t^{32} + 144539t^{33}\\
&& +\; 139430t^{34} + 131305t^{35} + 120931t^{36} + 108095t^{37}  + 93604t^{38} + 80199t^{39} + 68481t^{40}\\
&& +\; 55663t^{41} + 42067t^{42} + 27881t^{43} + 13597t^{44}
\eeq

\def\cprime{$'$}

\end{document}